\documentclass[11pt]{amsart}
\usepackage{amssymb,amsmath,amsthm}
\usepackage{latexsym}
\usepackage{epsfig}
\usepackage{graphicx}
\usepackage{geometry}
\newtheorem{theorem}{Theorem}[section]
\newtheorem{lemma}[theorem]{Lemma}
\newtheorem{remark}[theorem]{Remark}
\newtheorem{definition}[theorem]{Definition}

\newtheorem{corollary}[theorem]{Corollary}

\newtheorem{example}[theorem]{Example}

\newtheorem{observation}[theorem]{Observation}

\long\def\symbolfootnote[#1]#2{\begingroup%
\def\thefootnote{\fnsymbol{footnote}}\footnote[#1]{#2}\endgroup}

\begin{document}

\title{Annihilation of torsion in homology of finite $m$-AQ quandles}
\author{J\'ozef H. Przytycki, Seung Yeop Yang}

\thispagestyle{empty}

\begin{abstract}
It is a classical result in reduced homology of finite groups that the order of a group annihilates its homology. Similarly, we have proved that the torsion subgroup of rack and quandle homology of a finite quasigroup quandle is annihilated by its order. However, it does not hold for connected quandles in general. In this paper, we define an $m$-almost quasigroup ($m$-AQ) quandle which is a generalization of a quasigroup quandle and study annihilation of torsion in its rack and quandle homology groups.
\end{abstract}

\maketitle
\markboth{\hfil{\sc Annihilation of torsion in homology of finite $m$-AQ quandles}\hfil}
\ \
\tableofcontents

\section{Introduction}

A \emph{quandle} is an algebraic structure with a set $X$ and a binary operation $*:X \times X \rightarrow X$ satisfying the following axioms:
\begin{enumerate}
  \item (idempotency) $a*a=a$ for any $a \in X,$
  \item (invertibility) for each $b \in X$, $*_{b}:X \rightarrow X$ given by $*_{b}(x)=x*b$ is invertible (the inverse is denoted by $\overline{*}_{b}$),
  \item (right self-distributivity) $(a*b)*c=(a*c)*(b*c)$ for any $a,b,c \in X.$
\end{enumerate}

A set $X$ with a binary operation $*:X \times X  \rightarrow X$ that satisfies the right self-distributive property and invertibility is called a \emph{rack}.

Rack and quandle theory is motivated by knot theory. The axioms of a quandle correspond to the three Reidemeister moves \cite{Joy, Matv}. When we color an oriented knot diagram by a given magma $(X;*)$ consisting of a finite set $X$ and a binary operation $*:X \times X \rightarrow X$ with the convention of Figure $1.1$, we need the above conditions $(1),(2),$ and $(3)$ to be satisfied in order that Reidemeister moves preserve the number of the allowed colorings. Quandles can be used for classifying classical knots \cite{Joy, Matv}.

\begin{figure}[h]
\centerline{{\psfig{figure=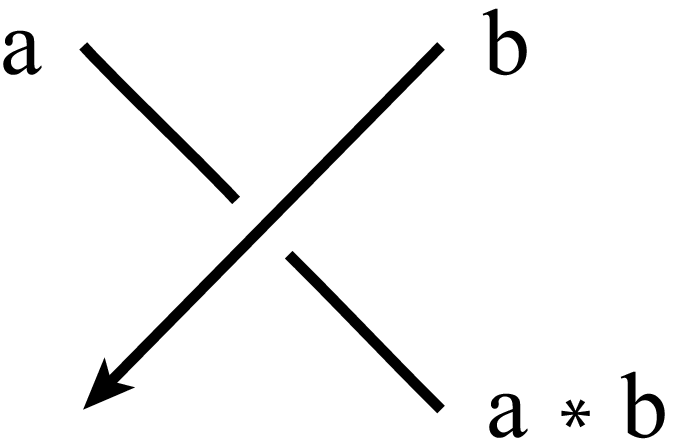,height=2.3cm}}}\ \\
\centerline{Figure 1.1; Quandle coloring}\ \\
\end{figure}

The followings are some examples of quandles:

\begin{example}
\begin{enumerate}
  \item \emph{Let $X$ be a set with a binary operation $a * b = a$ for any $a,b \in X$. Then $X$ is a quandle called a \emph{trivial quandle}.}
  \item \emph{A group $G$ with the conjugate operation $g * h = h^{-1}gh$ forms a quandle structure called a \emph{conjugate quandle}. More generally, any subset of $G$ which is closed under conjugation is a subquandle of a conjugate quandle $G.$}
  \item \cite{Tak} \emph{If $G$ is an abelian group, we define a quandle called a \emph{Takasaki quandle} (or \emph{kei}) of $G$, denoted by $T(G)$, by taking $a*b=2b-a$. Specially, if $G=\mathbb{Z}_{n}$, then we denote $T(\mathbb{Z}_{n})$ by $R_{n}$ called a \emph{dihedral quandle}.}
  \item \emph{Let $M$ be a module over the Laurent polynomial ring $\mathbb{Z}[t^{\pm1}]$. Then a quandle $M$ with the operation $a * b = ta + (1-t)b$ is said to be an \emph{Alexander quandle}.}
\end{enumerate}
\end{example}

Elementary quandle theory has been developed in a way similar to basic group theory.
A function $h:X \rightarrow Y$ between two quandles $(X;\ast)$ and $(Y;\ast^{'})$ is said to be a \emph{quandle homomorphism} if $h(a * b)=h(a)*^{'}h(b)$ for any $a, b \in X$. If a quandle homomorphism is invertible, then we call it a \emph{quandle isomorphism}. A \emph{quandle automorphism} is a quandle isomorphism from a quandle $X$ onto itself.

Recall that $*$ has an inverse binary operation $\overline{*}$, that is $(a*b)\overline{*}b=a=(a\overline{*}b)*b$ for any $a,b \in X.$ A \emph{subquandle} of a quandle $X$ is a subset $S$ of $X$ which is closed under $*$ and $\overline{*}$ operations.

\begin{observation} \label{Observation 1.2}
If $(X;*)$ is a finite quandle, then $S \subset X$ is a subquandle if it is closed under $*$ operation. It is the case because for any $a,b \in X$ the sequence $a*^{k}b$ where $a*^{k}b=(\cdots((a*b)*b)*\cdots)*b$ has to have repetitions, say $a*^{m}b=a*^{n}b$ for $0 \leq m < n,$ so $a*^{n-m}b=a.$ That is, $a\overline{*}b = a*^{n-m-1}b.$
\end{observation}

A quandle $X$ is said to be a \emph{quasigroup quandle} (or \emph{Latin quandle}) if it satisfies the quasigroup property, i.e. for any $a,b \in X$, the equation $a * x =b$ has a unique solution $x$. For each $a \in X$, the map $*_{b}:X \rightarrow X$ defined by $*_{b}(x)=x*b$ is a quandle automorphism. A subgroup of the quandle automorphism group $\textrm{Aut}(X)$ of a quandle $X$ generated by $*_{b}$ for every $b \in X$ is called a \emph{quandle inner automorphism group} denoted by $\textrm{Inn}(X).$ If the action of $\textrm{Inn}(X)$ on a quandle $X$ is transitive\footnote{It means that for any $a$ and $b$ in a quandle $X$ there are $x_{1},\ldots, x_{n} \in X$ so that $(\cdots((a*^{\varepsilon_{1}}x_{1})*^{\varepsilon_{2}}x_{2})*^{\varepsilon_{3}}\cdots)*^{\varepsilon_{n}}x_{n}=b$ where $*^{\varepsilon}=*$ or $\overline{*}.$ For a finite quandle $X$, by Observation \ref{Observation 1.2}, we can find $y_{1},\ldots, y_{k} \in X$ so that $(\cdots((a*y_{1})*y_{2})*\cdots)*y_{k}=b.$}, then we call $X$ a \emph{connected quandle}. Note that if a quandle $X$ has the quasigroup property, then it is a connected quandle. The converse, however, does not hold in general (see Example \ref{Example 1.3} $(3)$).

\begin{example} \label{Example 1.3}
\begin{enumerate}
  \item \emph{Dihedral quandles $R_{n}$ of odd order are quasigroup quandles.}
  \item \emph{An Alexander quandle $M$ is a quasigroup quandle if and only if $1-T$ is invertible.}
  \item \emph{The quandle QS(6) (or Rig quandle\footnote{L. Vendramin obtained a list of all connected quandles of order less than $48$ by using GAP, and they can be found in the GAP package Rig \cite{Ven}. There are $431$ connected quandles of order less than $36$ \cite{CSV}. These quandles are called \emph{Rig quandles}, and we denote $i$-th quandle of order $n$ in the list of Rig quandles by $Q(n,i).$ } $Q(6,2)$), which is the orbit of $4$-cycles in the symmetric group $S_{4}$ of order $24$ with the conjugate operation, is a connected quandle but not a quasigroup quandle (see \cite{CKS-1} and \cite{CKS-2}).}
\end{enumerate}
\end{example}

Rack homology theory was introduced by Fenn, Rourke, and Sanderson \cite{FRS}, and it was modified by Carter, Jelsovsky, Kamada, Langford, and Saito \cite{CJKLS} to define a (quandle) cocycle knot invariant. We recall the definition of rack and quandle homology based on \cite{CKS-2}, starting from the rack chain complex.

\begin{definition}
\emph{Let $C_{n}^{R}(X)$ be the free abelian group generated by n-tuples $(x_{1}, \ldots ,x_{n})$ of elements of a rack $X$, i.e. $C_{n}^{R}(X)=\mathbb{Z}X^{n}=(\mathbb{Z}X)^{\otimes n}$. We define a boundary homomorphism $\partial_{n}:C_{n}^{R}(X) \rightarrow C_{n-1}^{R}(X)$ by
$$\partial_{n} (x_{1}, \ldots , x_{n})=\sum\limits_{i=1}^{n}(-1)^{i}(d_{i}^{(*_{0})}-d_{i}^{(*)})(x_{1}, \ldots , x_{n})$$
where $d_{i}^{(*_{0})}(x_{1}, \ldots , x_{n})=(x_{1},\ldots,x_{i-1},x_{i+1},\ldots,x_{n})$ and\\
\hspace*{1.5cm} $d_{i}^{(*)}(x_{1}, \ldots , x_{n})=(x_{1}*x_{i},\ldots,x_{i-1}*x_{i},x_{i+1},\ldots,x_{n}).$\\
The face maps $d_{i}^{(*)}$ and $d_{i}^{(*_{0})}$ are illustrated in Figure $1.2.$\\
Then $(C_{n}^{R}(X),\partial_{n})$ is said to be a \emph{rack chain complex} of $X$.}
\end{definition}

\centerline{{\psfig{figure=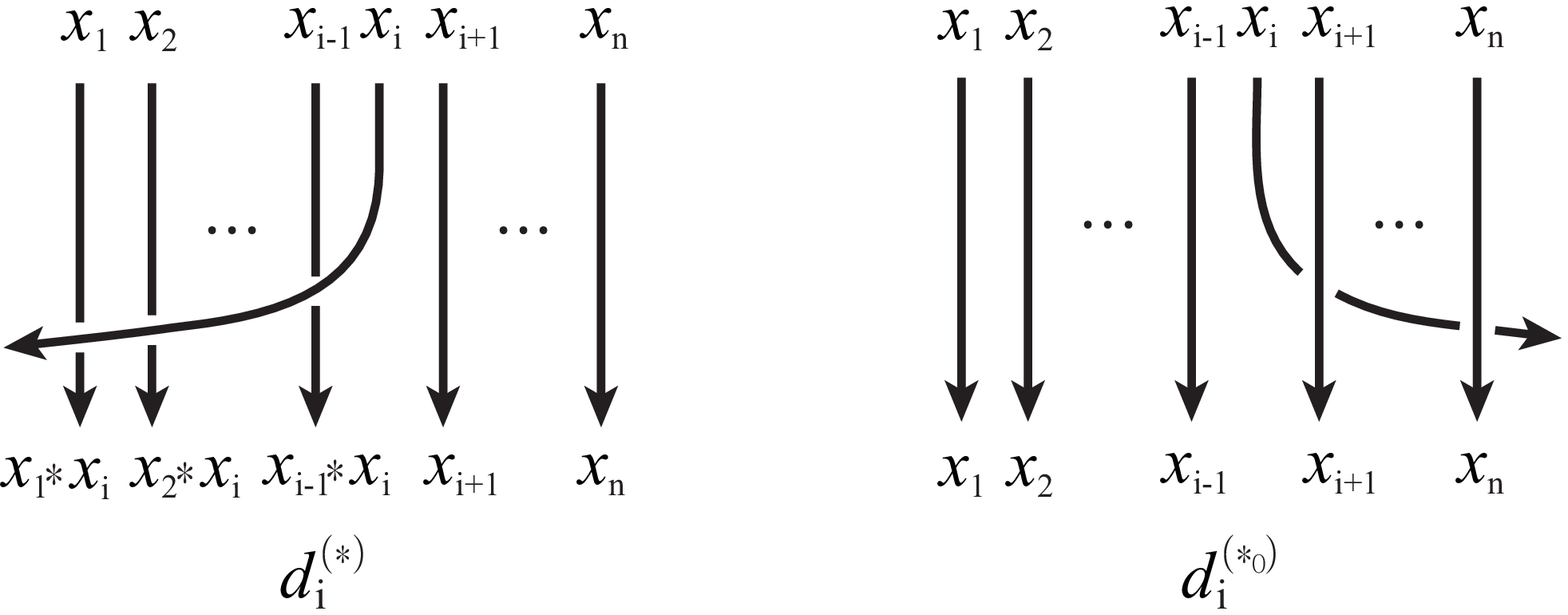,height=3.9cm}}}\ \\
\centerline{Figure 1.2; Graphical descriptions of face maps $d_{i}^{(*)}$ and $d_{i}^{(*_{0})}$}\ \\

\begin{definition}
\emph{For a quandle $X$, we consider the subgroup $C_{n}^{D}(X)$ of $C_{n}^{R}(X)$ generated by n-tuples $(x_{1}, \ldots ,x_{n})$ of elements of $X$ with $x_{i}=x_{i+1}$ for some $i=1,\ldots,n-1$. Then $(C_{n}^{D}(X),\partial_{n})$ is the subchain complex of a rack chain complex $(C_{n}^{R}(X),\partial_{n})$ and it is called the \emph{degenerate chain complex of $X$}. Then we take the quotient chain complex $(C_{n}^{Q}(X),\partial_{n})=(C_{n}^{R}(X)/C_{n}^{D}(X),\partial_{n})$ and call it the \emph{quandle chain complex}.}
\end{definition}

\begin{definition}
\emph{For an abelian group $G$, we define the chain complex $C_{*}^{W}(X;G)=C_{*}^{W}(X) \otimes G$ with $\partial = \partial \otimes \text{Id}$ for W=R, D, and Q. Then the \emph{$n$th rack, degenerate, and quandle homology groups of a quandle $X$ with coefficient group $G$} are respectively defined as
$$H_{n}^{W}(X;G)=H_{n}(C_{*}^{W}(X;G)) \hbox{~for W=R, D, and Q}.$$}
\end{definition}

For any finite rack or quandle, the free parts of rack, degenerate, and quandle homology groups have been completely obtained in \cite{CJKS,E-G,L-N}.

\begin{theorem} \cite{CJKS,E-G,L-N}
Let $\mathcal{O}$ be the set of orbits of a rack $X$ with respect to the action of $X$ on itself by right multiplication. Then
\begin{enumerate}\label{Theorem 1.7}
  \item \text{\emph{rank}}$H_{n}^{R}(X)=|\mathcal{O}|^{n}$ for a finite rack $X$,
  \item \text{\emph{rank}}$H_{n}^{Q}(X)=|\mathcal{O}|(|\mathcal{O}|-1)^{n-1}$ for a finite quandle $X$,
  \item \text{\emph{rank}}$H_{n}^{D}(X)=|\mathcal{O}|^{n}-|\mathcal{O}|(|\mathcal{O}|-1)^{n-1}$ for a finite quandle $X$.
\end{enumerate}
\end{theorem}

For every finite connected quandle $X,$ by Theorem \ref{Theorem 1.7} we have \textrm{rank}$H_{n}^{R}(X)=1$ for all $n,$ \textrm{rank}$H_{1}^{Q}(X)=1,$ and \textrm{rank}$H_{n}^{Q}(X)=0$ if $n \geq 2.$ Annihilation of torsion subgroups of rack and quandle homology of quasigroup quandles was addressed in \cite{P-Y}\footnote{Annihilation of torsion of homology of $R_{3}$ by $3$ was proven in \cite{N-P-1}. See \cite{P-Y} for the history of annihilation problem. }:

\begin{theorem} \cite{P-Y} \label{Theorem 1.8}
Let $Q$ be a finite quasigroup quandle. Then the torsion subgroup of $H_{n}^{R}(Q)$ is annihilated by $|Q|$.
\end{theorem}

\begin{corollary} \cite{P-Y} \label{Corollary 1.9}
The reduced quandle homology of a finite quasigroup quandle is annihilated by its order, i.e. $|Q|\widetilde{H}_{n}^{Q}(Q)=0$.
\end{corollary}

We can generalize the Theorem \ref{Theorem 1.8} to finite quasigroup racks.

\begin{corollary}
Let $X$ be a finite quasigroup rack. Then the torsion subgroup of $H_{n}^{R}(X)$ is annihilated by $|X|$.
\end{corollary}

\begin{proof}
It is well known that in a rack, $b$ and $b*b$ are functionally the same. Namely for any $a\in X$ we have
$$a*b= ((a \overline{*} b)*b)*b= a*(b*b).$$ From the quasigroup property we know that if $a*x=a*y$ then $x=y$ thus in a
quasigroup rack $b=b*b$ and the rack is a quandle.
\end{proof}

The $6$-elements quandle $QS(6)$ (or Rig quandle $Q(6,2)$, see Table \ref{Table 1} and Example \ref{Example 1.3} (3)) is a connected quandle, but it is not a quasigroup quandle and $H_{3}^{Q}(QS(6))=\mathbb{Z}_{24}$ (see \cite{CKS-1}). Thus $6$ does not annihilate $\text{tor}H_{3}(QS(6))$; this shows that Theorem \ref{Theorem 1.8} and Corollary \ref{Corollary 1.9} do not hold for connected quandles in general. We show that the torsion subgroup of rack and quandle homology of $QS(6)$ is annihilated by the order of its quandle inner automorphism group ($24$ in this case) in Theorem \ref{Theorem 3.2}.

\begin{table}[h]
  \centering
  \caption{A quandle $QS(6)$}\label{Table 1}
  \begin{tabular}{c|cccccc}
 $\ast$ & 1 & 2 & 3 & 4 & 5 & 6 \\
 \hline
 1 & 1 & 1 & 6 & 5 & 3 & 4 \\
 2 & 2 & 2 & 5 & 6 & 4 & 3 \\
 3 & 5 & 6 & 3 & 3 & 2 & 1 \\
 4 & 6 & 5 & 4 & 4 & 1 & 2 \\
 5 & 4 & 3 & 1 & 2 & 5 & 5 \\
 6 & 3 & 4 & 2 & 1 & 6 & 6

\end{tabular}
\end{table}

\section{$m$-almost quasigroup ($m$-AQ) quandles}

An $m$-almost quasigroup quandle ($m$-AQ quandle, in short) is a generalization of a quasigroup quandle. A formal definition follows.

\begin{definition}
\emph{A quandle $X$ is said to be \emph{$m$-almost quasigroup} if it satisfies the following conditions:
\begin{enumerate}
  \item for each $a \in X$, the stabilizer set $S_{a}=\{ x \in X | a*x=a \}$ has order $m$, say $S_{a}=\{a=a^{(1)}, a^{(2)}, \cdots, a^{(m)} \},$
  \item for any $b \in X \setminus S_{a}$, the equation $a*x=b$ has a unique solution.
\end{enumerate}}
\end{definition}

If $m=1$ above, then $X$ is a quasigroup quandle. Note that if $X$ is finite and $m \geq 2$, then the equation $a*x = a^{(k)}$ has no solution if $2 \leq k \leq m.$

The following are some examples of $m$-almost quasigroup quandles:

\begin{example}
\begin{enumerate}
  \item \emph{Every finite trivial quandle $X$ is an $|X|$-almost quasigroup quandle.}
  \item \emph{The quandle QS(6) (or Rig quandle $Q(6,2)$) is a $2$-almost quasigroup quandle.}
  \item \emph{A quandle composed of $2$-cycles in the symmetric group $S_{n}$ and the conjugate operation is an $(\binom{n-2}{2}+1)$-almost quasigroup quandle if $n \geq 2$ (for $2$-cycles in $S_{n},$ $(i,j)*(k,l)=(i,j)$ if and only if $\{i,j\}=\{k,l\}$ or $\{i,j\} \cap \{k,l\}=\emptyset$). In particular, Rig quandles $Q(6,1),$ $Q(10,1),$ $Q(15,7),$ $Q(21,9),$ and $Q(28,13)$ are the cases of $n=4,$ $n=5,$ $n=6,$ $n=7,$ and $n=8,$ respectively.}
  \item \emph{The Rig quandle $Q(15,2)$ which is composed of elements of $S_{5}$ with cycle partition $(2,2,1)$ is a $3$-almost quasigroup quandle.}
\end{enumerate}
\end{example}

We check some basic properties of $m$-almost quasigroup quandles:

\begin{lemma}\label{Lemma 2.3}
Suppose that $X$ is an $m$-almost quasigroup quandle.
\begin{enumerate}
  \item For each $a\in X$, the stabilizer set $S_{a}$ is a subquandle of $X$.
  \item For any $a,b \in X$, $S_{a}*b=S_{a*b}$ where $S_{a}*b=\{a*b, a^{(2)}*b,\cdots,a^{(m)}*b\}.$
  \item If $X$ is finite and $m \leq 3,$ then every stabilizer set is a trivial subquandle of $X$.
  \item If $X$ is nontrivial and every stabilizer set is a trivial subquandle of $X,$ then $X$ is connected.
\end{enumerate}
\end{lemma}

\begin{proof}
$(1)$ For $a \in X$, we consider the stabilizer set $S_{a}=\{a=a^{(1)}, a^{(2)}, \cdots, a^{(m)} \}.$  Then $a=a*a^{(j)}=(a*a^{(i)})*a^{(j)}=(a*a^{(j)})*(a^{(i)}*a^{(j)})=a*(a^{(i)}*a^{(j)})$ so that $a^{(i)}*a^{(j)} \in S_{a}$ for any $1 \leq i, j \leq m.$ Moreover, $a*a^{(j)}=a$ implies that $a = a\overline{*}a^{(j)}$ so that $a = (a*a^{(i)})\overline{*}a^{(j)}=(a\overline{*}a^{(j)})*(a^{(i)}\overline{*}a^{(j)})=a*(a^{(i)}\overline{*}a^{(j)}).$ Therefore we similarly have $a^{(i)}\overline{*}a^{(j)} \in S_{a}$ for any $1 \leq i, j \leq m.$ That is, $S_{a}$ is closed under both operations $*$ and $\overline{*},$ hence $S_{a}$ is a subquandle of $X.$

$(2)$ For any $i \in \{1, 2, \cdots m \},$ $(a*b)*(a^{(i)}*b)=(a*a^{(i)})*b=a*b$, and this implies $S_{a}*b \subset S_{a*b}.$ Since $*_{b}$ is bijective, $|S_{a}*b|=m=|S_{a*b}|$ therefore $S_{a}*b=S_{a*b}.$

$(3)$ Let $a \in X$, then the stabilizer set $S_{a}$ is a subquandle of $X$ by Lemma \ref{Lemma 2.3} $(1).$ If $m \leq 2,$ then $S_{a}$ is clearly trivial, so we are done. Assume that $m=3$ and denote $S_{a}=\{a,b,c\}.$ Since $a*c=a$ and $c*c=c,$ by the invertibility condition of a quandle we have $b*c=b,$ i.e. $c\in S_{b}.$ Then $b*a$ should be $b$ or $c$ because $S_{a}$ is closed under $*$ and $a*a=a.$ Since $X$ is finite, the second axiom of the definition of an $m$-almost quasigroup quandle implies that for any $y \in S_{b}$ the equation $b*x=y$ has no solution. Therefore, the equation $b*x=c$ has no solution, so $b*a=b.$ Then $c*a=c$ and $c*b=c$ by the invertibility condition of a quandle, therefore $S_{a}$ is a trivial subquandle of $X.$

$(4)$ Since $X$ is a nontrivial $m$-almost quasigroup quandle, $m < |X|.$ Let $a, b \in X.$ If $b$ is not contained in $S_{a},$ then there exists a unique solution $y \in X$ such that $a*y=b$ by the second axiom of the definition of an $m$-almost quasigroup quandle. Suppose that $b \in S_{a}.$ Then the triviality of $S_{a}$ implies that $S_{a}=S_{b}.$ Since $m < |X|,$ there is an element $c \in X$ such that $c \in X \setminus S_{a}.$ Then we can find a unique element $y_{1} \in X$ so that $a*y_{1}=c.$ We easily check that $b \in X \setminus S_{c}$ (because if $b \in S_{c},$ then since $S_{c}$ is trivial, we have the equation $b*c=b$, i.e. $c \in S_{b}=S_{a},$ and this contradicts $c \in X \setminus S_{a}$). Then there is a unique element $y_{2} \in X$ so that $c*y_{2}=b,$ and then we have $(a*y_{1})*y_{2}=c*y_{2}=b.$ Therefore the quandle $X$ is connected.
\end{proof}

Notice that the stabilizer set $S_{a}=\{a=a^{(1)}, a^{(2)}, \cdots, a^{(m)} \}$ is a trivial subquandle of a quandle $X$ if and only if $S_{a}=S_{a^{(k)}}$ for any $k.$ In general the equality $S_{a}=S_{a^{(k)}}$ does not have to hold; see the Table \ref{Table 2}.

Carter, Elhamdadi, Nikiforou, and Saito \cite{CENS} introduced an abelian extension theory of quandles, and a generalization to extensions with a dynamical cocycle was defined by Andruskiewitsch and Gra\~na \cite{AG}. We review the definition of quandle extension by a dynamical cocycle after \cite{CHNS}.

\begin{definition}\cite{AG, CHNS}
\emph{Let $X$ be a quandle and $S$ be a non-empty set. Let $\alpha: X \times X \rightarrow \text{Fun}(S\times S,S)=S^{S\times S}$ be a function, so that for $a, b \in X$ and $s,t \in S$ we have $\alpha_{a, b}(s,t) \in S.$ Then $S \times X$ is a quandle by the operation $(s, a)\ast (t, b)= (\alpha_{a, b}(s,t),a \ast b),$ where $a \ast b$ denotes the quandle operation in $X$, if and only if $\alpha$ satisfies the following conditions:
\begin{enumerate}
  \item $\alpha_{a, a}(s,s)=s$ for all $a \in X$ and $s \in S,$
  \item $\alpha_{a, b}(-,t):S\rightarrow S$ is a bijection for all $a, b \in X$ and for all $t \in S,$
  \item $\alpha_{a \ast b, c}(\alpha_{a, b}(s,t),u)=\alpha_{a \ast c, b \ast c}(\alpha_{a, c}(s,u),\alpha_{b, c}(t,u))$ for all $a, b, c \in X$ and $s, t, u \in S.$
\end{enumerate}
Such a function $\alpha$ is called a \emph{dynamical quandle cocycle}. The quandle constructed above is denoted by $S\times_{\alpha}X,$ and is called the \emph{extension} of $X$ by a dynamical cocycle $\alpha.$}
\end{definition}

Unlike quasigroup quandles, $m$-almost quasigroup quandles are not connected in general. We obtain interesting examples of non-connected $m$-almost quasigroup quandles from quandle extensions of quasigroup quandles and trivial quandles.

\begin{example}
\emph{Let $(X;*)$ be a finite quasigroup quandle and $T_{n}$ the trivial quandle of order $n$. Then the dynamical quandle cocycle extension $X\times_{\alpha}T_{n}$ is an $((n-1)|X|+1)$-almost quasigroup quandle (but not a connected quandle if $n \geq 2$) where the dynamical quandle cocycle $\alpha: T_{n} \times T_{n} \rightarrow X^{X\times X}$ is defined by $\alpha_{a,b}(s,t)= s*t$ if $a=b$ and $\alpha_{a,b}(s,t)= s$ if $a\neq b.$}
\end{example}

\begin{table}[h]
  \centering
  \caption{A $4$-almost quasigroup quandle which is not connected; $R_{3}\times_{\alpha} T_{2}$}\label{Table 2}
  \begin{tabular}{c|cccccc}
 $\ast$ & 1 & 2 & 3 & 4 & 5 & 6 \\
 \hline
 1 & 1 & 3 & 2 & 1 & 1 & 1 \\
 2 & 3 & 2 & 1 & 2 & 2 & 2 \\
 3 & 2 & 1 & 3 & 3 & 3 & 3 \\
 4 & 4 & 4 & 4 & 4 & 6 & 5 \\
 5 & 5 & 5 & 5 & 6 & 5 & 4 \\
 6 & 6 & 6 & 6 & 5 & 4 & 6

\end{tabular}
\end{table}

We computed that $H_{2}^{R}(R_{3}\times_{\alpha} T_{2})=\mathbb{Z}^{4},$ $H_{3}^{R}(R_{3}\times_{\alpha} T_{2})=\mathbb{Z}^{8}\oplus\mathbb{Z}_{3}^{2},$ $H_{4}^{R}(R_{3}\times_{\alpha} T_{2})=\mathbb{Z}^{16}\oplus\mathbb{Z}_{3}^{8},$ $H_{2}^{Q}(R_{3}\times_{\alpha} T_{2})=\mathbb{Z}^{2},$ $H_{3}^{Q}(R_{3}\times_{\alpha} T_{2})=\mathbb{Z}^{2}\oplus\mathbb{Z}_{3}^{2},$ and $H_{4}^{Q}(R_{3}\times_{\alpha} T_{2})=\mathbb{Z}^{2}\oplus\mathbb{Z}_{3}^{6}.$

\section{Annihilation of rack and quandle homology groups of $m$-almost quasigroup quandles}

Our main result of the paper is Theorem \ref{Theorem 3.2}. We start from the important preparatory Lemma \ref{Lemma 3.1}.

Let $X$ be an $m$-almost quasigroup quandle and ${\bf x}=(x_{1}, \cdots, x_{n}) \in X^{n}$. We consider two chain maps $g_{1}^{j},g_{2}^{j}:C_{n}^{R}(X) \rightarrow C_{n}^{R}(X)$ given by
$$g_{1}^{j}({\bf x})=\sum\limits_{k=1}^{m}( x_{j}^{(k)}, \ldots, x_{j}^{(k)}, x_{j}, x_{j+1},\ldots, x_{n} ),$$
$$g_{2}^{j}({\bf x})=\sum\limits_{k=1}^{m}( x_{j}^{(k)}, \ldots, x_{j}^{(k)}, x_{j}^{(k)}, x_{j+1},\ldots, x_{n} )$$
for $1 \leq j \leq n.$ Note that the chain map $g_{1}^{1}$ is equal to $m\text{Id}.$

\begin{lemma}\label{Lemma 3.1}
Let $X$ be a finite $m$-almost quasigroup quandle. Suppose that for each $a \in X$ the stabilizer set $S_{a}$ is a trivial subquandle of $X$.
Then the chain maps $(|X|-m)g_{1}^{j}$ and $(|X|-m)g_{2}^{j}$ are chain homotopic for each $1 \leq j \leq n.$
\end{lemma}

\begin{proof}
We define a chain map $g_{0}^{j}:C_{n}^{R}(X) \rightarrow C_{n}^{R}(X)$ by
$$g_{0}^{j}({\bf x})=\sum\limits_{k=1}^{m}( x_{j}, \ldots, x_{j}, x_{j}^{(k)}, x_{j+1},\ldots, x_{n} )$$
for $1 \leq j \leq n.$ Notice that $g_{0}^{1}=g_{2}^{1}.$\\

We consider a chain homotopy $G_{n}^{j}:C_{n}^{R}(X) \rightarrow C_{n+1}^{R}(X)$ given by for $1 \leq j \leq n$
$$G_{n}^{j}({\bf x})=\sum\limits_{k=1}^{m}\sum\limits_{y \in X}\{( x_{j}^{(k)}, \ldots, x_{j}^{(k)}, x_{j}, y, x_{j+1},\ldots, x_{n} ) -( x_{j}, \ldots, x_{j}, x_{j}^{(k)},y, x_{j+1},\ldots, x_{n} )\}.$$

If $i \leq j-1,$ then by the idempotence condition of a quandle we have
$$d_{i}^{(\ast)}G_{n}^{j}({\bf x})=\sum\limits_{k=1}^{m}\sum\limits_{y \in X}\{( x_{j}^{(k)}, \ldots, x_{j}^{(k)}, x_{j}, y, x_{j+1},\ldots, x_{n} ) -( x_{j}, \ldots, x_{j}, x_{j}^{(k)},y, x_{j+1},\ldots, x_{n} )\}$$
so that the formula above does not depend on the binary operation $*$, therefore $(d_{i}^{(\ast_{0})}-d_{i}^{(\ast)})G_{n}^{j}=0.$\\

If $i = j,$ then the formula again does not depend on the binary operation $*$ because the stabilizer set $S_{x_{j}}$ is trivial
$$d_{i}^{(\ast)}G_{n}^{j}({\bf x})=\sum\limits_{k=1}^{m}\sum\limits_{y \in X}\{( x_{j}^{(k)}, \ldots, x_{j}^{(k)}, y, x_{j+1},\ldots, x_{n} ) -( x_{j}, \ldots, x_{j}, y, x_{j+1},\ldots, x_{n} )\}$$
thus $(d_{i}^{(\ast_{0})}-d_{i}^{(\ast)})G_{n}^{j}=0.$\\

If $i = j+1$, then since the stabilizer set $S_{x_{j}}$ is trivial and $X$ satisfies the $m$-almost quasigroup property, we have $\sum\limits_{y \in X}(x_{j}^{(k)}*y)=m(x_{j}^{(k)})+\sum\limits_{y \in X \setminus S_{x_{j}}}(x_{j}^{(k)}*y)=m(x_{j}^{(k)})+\sum\limits_{y \in X \setminus S_{x_{j}}}(y)$ and therefore
$$d_{i}^{(\ast)}G_{n}^{j}({\bf x})=\sum\limits_{k=1}^{m}\{m( x_{j}^{(k)}, \ldots, x_{j}^{(k)}, x_{j}, x_{j+1},\ldots, x_{n} )+ \sum\limits_{y \in X \setminus S_{x_{j}}}( x_{j}^{(k)}*y, \ldots, x_{j}^{(k)}*y, x_{j}*y, x_{j+1},\ldots, x_{n} ) $$
$$ -m( x_{j}, \ldots, x_{j}, x_{j}^{(k)}, x_{j+1},\ldots, x_{n} ) -\sum\limits_{y \in X \setminus S_{x_{j}}}( x_{j}*y, \ldots, x_{j}*y, x_{j}^{(k)}*y, x_{j+1},\ldots, x_{n} )\} $$
$$=m\sum\limits_{k=1}^{m}\{( x_{j}^{(k)}, \ldots, x_{j}^{(k)}, x_{j}, x_{j+1},\ldots, x_{n} ) -  ( x_{j}, \ldots, x_{j}, x_{j}^{(k)}, x_{j+1},\ldots, x_{n} )\}.$$
The part corresponding to the sum over $y \in X \setminus S_{x_{j}}$ cancels out because for any $x_{j},$ $x_{j}^{(k)} \in S_{x_{j}},$ and $y \in X \setminus S_{x_{j}},$ there are unique $l \in \{1,\ldots, m \}$ and $z \in X \setminus S_{x_{j}}$ so that $(x_{j}^{(k)}*y,x_{j}*y)=(x_{j}*z,x_{j}^{(l)}*z).$

Then we have
$$(d_{i}^{(\ast_{0})}-d_{i}^{(\ast)})G_{n}^{j}({\bf x})=(|X|-m)\sum\limits_{k=1}^{m}\{( x_{j}^{(k)}, \ldots, x_{j}^{(k)}, x_{j}, x_{j+1},\ldots, x_{n} ) -( x_{j}, \ldots, x_{j}, x_{j}^{(k)}, x_{j+1},\ldots, x_{n} )\}.$$

Lastly, if $j+2 \leq i \leq n+1$, then by the invertibility condition of a quandle we have $\sum\limits_{y \in X}(y*x_{i-1})=\sum\limits_{y \in X}(y)$ and therefore
$$d_{i}^{(\ast)}G_{n}^{j}({\bf x})=\sum\limits_{k=1}^{m}\sum\limits_{y \in X}\{( x_{j}^{(k)}*x_{i-1}, \ldots, x_{j}^{(k)}*x_{i-1}, x_{j}*x_{i-1}, y, x_{j+1}*x_{i-1},\ldots,x_{i-2}*x_{i-1},x_{i},\cdots, x_{n} )$$ $$-( x_{j}*x_{i-1}, \ldots, x_{j}*x_{i-1}, x_{j}^{(k)}*x_{i-1}, y, x_{j+1}*x_{i-1},\ldots,x_{i-2}*x_{i-1},x_{i},\cdots, x_{n} )\}.$$

On the other hand, if $i \leq j$, then
$$G_{n-1}^{j}d_{i}^{(\ast)}({\bf x})=\sum\limits_{k=1}^{m}\sum\limits_{y \in X}\{( x_{j+1}^{(k)}, \ldots, x_{j+1}^{(k)}, x_{j+1}, y, x_{j+2}, \ldots, x_{n} )-( x_{j+1}, \ldots, x_{j+1}, x_{j+1}^{(k)}, y, x_{j+2},\ldots, x_{n} )\},$$
so that the formula above does not depend on the binary operation $*$, and $G_{n-1}^{j}(d_{i}^{(\ast_{0})}-d_{i}^{(\ast)})=0.$\\

If $j+1 \leq i$, then
$$G_{n-1}^{j}d_{i}^{(\ast)}({\bf x})=\sum\limits_{k=1}^{m}\sum\limits_{y \in X}\{((x_{j} \ast x_{i})^{(k)},\ldots,(x_{j} \ast x_{i})^{(k)},x_{j} \ast x_{i}, y, x_{j+1} \ast x_{i},\ldots, x_{i-1} \ast x_{i},x_{i+1},\ldots,x_{n}) $$
$$ -(x_{j} \ast x_{i},\ldots,x_{j} \ast x_{i},(x_{j} \ast x_{i})^{(k)}, y, x_{j+1} \ast x_{i},\ldots, x_{i-1} \ast x_{i},x_{i+1},\ldots,x_{n})\}.$$
Notice that $d_{i+1}^{(\ast_{0})}G_{n}^{j}=G_{n-1}^{j}d_{i}^{(\ast_{0})}$ and by Lemma \ref{Lemma 2.3} $(2)$ $d_{i+1}^{(\ast)}G_{n}^{j}=G_{n-1}^{j}d_{i}^{(\ast)}$ if $j+1 \leq i \leq n.$\\

Hence we have the following equality:
$$\partial_{n+1}G_{n}^{j}({\bf x})+G_{n-1}^{j}\partial_{n}({\bf x})=(-1)^{j+1}(|X|-m)(g_{1}^{j}({\bf x})-g_{0}^{j}({\bf x})),$$
that means $(|X|-m)g_{1}^{j}$ and $(|X|-m)g_{0}^{j}$ are chain homotopic for each $1 \leq j \leq n.$\\

We next consider a chain homotopy $F_{n}^{j}:C_{n}^{R}(X) \rightarrow C_{n+1}^{R}(X)$ defined by
$$F_{n}^{j}({\bf x})=\sum\limits_{k=1}^{m}\sum\limits_{y \in X}\{( x_{j}, \ldots, x_{j}, y, x_{j}^{(k)}, x_{j+1},\ldots, x_{n} ) -( x_{j}^{(k)}, \ldots, x_{j}^{(k)},y, x_{j}^{(k)}, x_{j+1},\ldots, x_{n} )\}$$
for $2 \leq j \leq n.$ Then by a similar calculation as above, we have the following equality:
$$\partial_{n+1}F_{n}^{j}({\bf x})+F_{n-1}^{j}\partial_{n}({\bf x})=(-1)^{j}(|X|-m)(g_{0}^{j}({\bf x})-g_{2}^{j}({\bf x})),$$
hence $(|X|-m)g_{0}^{j}$ is chain homotopic to $(|X|-m)g_{2}^{j}$ for each $2 \leq j \leq n.$\\

Therefore, we obtain the following sequence of chain homotopic chain maps
$$(|X|-m)g_{1}^{j} \simeq (|X|-m)g_{0}^{j} \simeq (|X|-m)g_{2}^{j}$$
for each $1 \leq j \leq n.$
\end{proof}

We now study annihilation of rack and quandle homology groups of $m$-almost quasigroup quandles under the same assumption as in Lemma \ref{Lemma 3.1} that every stabilizer set $S_{a}$ is a trivial subquandle of $X$.

\begin{theorem} \label{Theorem 3.2}
Let $X$ be a finite $m$-almost quasigroup quandle. Suppose that for every $a \in X$ the stabilizer set $S_{a}$ is a trivial subquandle of $X$. Then the torsion subgroup of $H_{n}^{R}(X)$ is annihilated by $m(\emph{lcm}(|X|,|X|-m)).$
\end{theorem}

\begin{proof}
If $X$ is a trivial quandle, i.e. $m=|X|$, then we are done because the rack homology group of a trivial quandle does not have any torsion elements. Assume now that $X$ is a nontrivial quandle. Then by Lemma \ref{Lemma 2.3} $(4),$ the quandle $X$ is connected. For each $j \in \{1,2,\cdots,n\},$ we define the symmetrizer chain map $g_{s}^{j}:C_{n}^{R}(X) \rightarrow C_{n}^{R}(X)$ by $g_{s}^{j}({\bf x})=\sum\limits_{y \in X}(y,\cdots,y,x_{j+1},\cdots,x_{n}).$

We first prove that $mg_{s}^{j}-|X|g_{2}^{j}$ for $j \in \{1,\cdots,n\}$ and $mg_{s}^{j-1}-|X|g_{1}^{j}$ for $j \in \{2,\cdots,n\}$ are null-homotopic by using chain homotopies $D_{n}^{j},E_{n}^{j}:C_{n}^{R}(X) \rightarrow C_{n+1}^{R}(X)$, respectively, given by
$$D_{n}^{j}({\bf x})=\sum\limits_{k=1}^{m}\sum\limits_{y \in X}(x_{j}^{(k)},\cdots,x_{j}^{(k)},x_{j}^{(k)},y,x_{j+1},\cdots,x_{n}) \text{~for~} 1 \leq j \leq n ,$$
$$E_{n}^{j}({\bf x})=\sum\limits_{k=1}^{m}\sum\limits_{y \in X}(x_{j}^{(k)},\cdots,x_{j}^{(k)},y,x_{j},x_{j+1},\cdots,x_{n}) \text{~for~} 2 \leq j \leq n .$$

When $i \leq j,$ by the idempotence condition of a quandle we obtain the following equation:
$$d_{i}^{(\ast)}D_{n}^{j}({\bf x})=\sum\limits_{k=1}^{m}\sum\limits_{y \in X}( x_{j}^{(k)}, \ldots, x_{j}^{(k)},y, x_{j+1},\ldots, x_{n} ).$$
Since the above formula does not depend on the binary operation $\ast,$ $(d_{i}^{(\ast_{0})}-d_{i}^{(\ast)})D_{n}^{j}=0.$\\

We check that $\sum\limits_{y \in X}(x_{j}^{(k)}*y)=m(x_{j}^{(k)})+\sum\limits_{y \in X \setminus S_{x_{j}}}(y)$ because the stabilizer set $S_{x_{j}}$ is trivial and $X$ satisfies the $m$-almost quasigroup property so that $\sum\limits_{k=1}^{m}\sum\limits_{y \in X}(x_{j}^{(k)}*y)=m\sum\limits_{y \in X}(y).$

Thus, if $i = j+1,$ then
$$d_{i}^{(\ast)}D_{n}^{j}({\bf x})=m\sum\limits_{y \in X}(y,\cdots,y,x_{j+1},\cdots, x_{n} )$$
therefore, $(d_{i}^{(\ast_{0})}-d_{i}^{(\ast)})D_{n}^{j}=|X|g_{2}^{j}-mg_{s}^{j}.$

Finally, if $j+2 \leq i \leq n+1$, then we obtain $\sum\limits_{y \in X}(y*x_{i-1})=\sum\limits_{y \in X}(y)$ from the invertibility condition of a quandle so that
$$d_{i}^{(\ast)}D_{n}^{j}({\bf x})=\sum\limits_{k=1}^{m}\sum\limits_{y \in X}( x_{j}^{(k)}*x_{i-1}, \ldots, x_{j}^{(k)}*x_{i-1},y, x_{j+1}*x_{i-1},\ldots,x_{i-2}*x_{i-1},x_{i},\cdots, x_{n} ).$$

On the other hand, when $i \leq j$ we have the following equation:
$$D_{n-1}^{j}d_{i}^{(\ast)}({\bf x})=\sum\limits_{k=1}^{m}\sum\limits_{y \in X}(x_{j+1}^{(k)}, \ldots, x_{j+1}^{(k)},y,x_{j+2},\ldots, x_{n}),$$
hence $D_{n-1}^{j}(d_{i}^{(\ast_{0})}-d_{i}^{(\ast)})=0$ because the formula above does not depend on the operation $*.$\\

If $j+1 \leq i$, then we have
$$D_{n-1}^{j}d_{i}^{(\ast)}({\bf x})=\sum\limits_{k=1}^{m}\sum\limits_{y \in X}((x_{j} \ast x_{i})^{(k)},\ldots,(x_{j} \ast x_{i})^{(k)},y, x_{j+1} \ast x_{i},\ldots, x_{i-1} \ast x_{i},x_{i+1},\ldots,x_{n}).$$

Note that if $j+1 \leq i \leq n,$ then $d_{i+1}^{(\ast_{0})}D_{n}^{j}=D_{n-1}^{j}d_{i}^{(\ast_{0})}$ and by Lemma \ref{Lemma 2.3} $(2)$ $d_{i+1}^{(\ast)}D_{n}^{j}=D_{n-1}^{j}d_{i}^{(\ast)}.$ \\

Finally, we obtain the following equality:
$$\partial_{n+1}D_{n}^{j}({\bf x})+D_{n-1}^{j}\partial_{n}({\bf x})=(-1)^{j+1}(|X|g_{2}^{j}({\bf x})-mg_{s}^{j}({\bf x})),$$
which means the chain maps $|X|g_{2}^{j}$ and $mg_{s}^{j}$ are chain homotopic for each $1 \leq j \leq n.$

By using the chain homotopy $E_{n}^{j}:C_{n}^{R}(X) \rightarrow C_{n+1}^{R}(X)$ and a similar calculation as above, we also have the following equality:
$$\partial_{n+1}E_{n}^{j}({\bf x})+E_{n-1}^{j}\partial_{n}({\bf x})=(-1)^{j}(|X|g_{1}^{j}({\bf x})-mg_{s}^{j-1}({\bf x})),$$
so that the chain maps $|X|g_{1}^{j}$ and $mg_{s}^{j-1}$ are chain homotopic for each $2 \leq j \leq n.$\\

Then from the above calculations and by Lemma \ref{Lemma 3.1}, we have a sequence of chain homotopic chain maps
$$N\text{Id} \simeq \frac{N}{m}g_{2}^{1} \simeq \frac{N}{|X|}g_{s}^{1} \simeq \frac{N}{m}g_{1}^{2} \simeq \frac{N}{m}g_{2}^{2} \simeq \frac{N}{|X|}g_{s}^{2} \simeq \cdots \simeq \frac{N}{m}g_{1}^{n} \simeq \frac{N}{m}g_{2}^{n} \simeq \frac{N}{|X|}g_{s}^{n}$$
where $N=m(\text{lcm}(|X|,|X|-m)).$ We therefore obtain the same induced homomorphisms\\
$N\text{Id}=(\frac{N}{|X|}g_{s}^{n})_{*}:H_{n}^{R}(X) \rightarrow H_{n}^{R}(X).$ Recall that since $X$ is connected, the free part of the $n$th rack homology group of $X$ is $\mathbb{Z}$ generated by $(y,\ldots,y)$ for $y \in X.$ Hence $N\textrm{tor}(H_{n}^{R}(X))=0$ for any dimension $n.$
\end{proof}

\begin{corollary}\label{Corollary 3.3}
Let $X$ be as in Theorem \ref{Theorem 3.2}. If $X$ is nontrivial, then the reduced quandle homology of $X$ is annihilated by $N=m(\emph{lcm}(|X|,|X|-m))$, i.e. $N\widetilde{H}_{n}^{Q}(X)=0$.
\end{corollary}

\begin{proof}
The rack homology of a quandle splits into degenerate homology and quandle homology, i.e. $H_{n}^{R}(X)=H_{n}^{D}(X) \oplus H_{n}^{Q}(X) $ (see \cite{L-N}), hence the torsion of $H_{n}^{Q}(X)$ is annihilated by $N$ by Theorem \ref{Theorem 3.2}. Furthermore, since $X$ is finite and nontrivial, $X$ is a finite connected quandle by Lemma \ref{Lemma 2.3} $(4).$ Then $\textrm{rank}(H_{n}^{Q}(X))=0$ for $n \geq 2$ and $\textrm{rank}(H_{1}^{Q}(X))=1.$ Therefore, the reduced quandle homology of $X$ is a torsion group annihilated by $N$.
\end{proof}

Corollary \ref{Corollary 3.4} is immediate from Lemma \ref{Lemma 2.3} $(3)$ and Theorem \ref{Theorem 3.2}.

\begin{corollary} \label{Corollary 3.4}
Let $X$ be a finite $m$-almost quasigroup quandle where $m \leq 3$. Then the torsion subgroup of $H_{n}^{R}(X)$ is annihilated by $N=m(\emph{lcm}(|X|,|X|-m)).$
\end{corollary}

The connected quandle $QS(6)$ (which is not a quasigroup quandle) is used to show that for any connected quandles the torsion subgroup of its rack and quandle homology can not be annihilated by its order in general. However, we obtain the best solution for the quandle $QS(6)$ from Theorem \ref{Theorem 3.2} and Corollary \ref{Corollary 3.3}.

\begin{example}
\emph{The least upper bound for orders of the torsion elements in $H_{n}^{R}(QS(6))$ is $|\textrm{Inn}(QS(6))|.$ For every dimension $n$, in particular, $|\textrm{Inn}(QS(6))|\widetilde{H}_{n}^{Q}(QS(6))=0.$\\
Since $QS(6)$ is a $2$-almost quasigroup quandle of order $6,$ $24$ is an upper bound for orders of the torsion elements in $H_{n}^{R}(QS(6))$ by Theorem \ref{Theorem 3.2}. Then $\textrm{Inn}(QS(6)) \cong S_{4}$ and $H_{3}^{Q}(QS(6))=\mathbb{Z}_{24}$ (see \cite{CKS-1}) implies that $|\textrm{Inn}(QS(6))|=|S_{4}|=24$ is the smallest number which annihilates $\textrm{tor}H_{n}^{R}(QS(6))$.}
\end{example}

\begin{remark}
The order of a quandle inner automorphism group is not always the least upper bound for orders of the torsion elements in the rack or quandle homology groups of the quandle. For example, the $3$-almost quasigroup quandle of order $12$ (Rig quandle $Q(12,10)$) given by the Table \ref{Table 3} (see \cite{Cla}) has $\emph{Inn}(Q(12,10))$ of order $216$ while Theorem \ref{Theorem 3.2} shows that $108$ annihilates the torsion of homology.
\end{remark}

\begin{example}
\emph{According to \cite{Ven} there are seven connected quandles of $15$ elements. Five of them are quasigroup quandles while the Rig quandle $Q(15,2)$ is a $3$-almost quasigroup quandle and the Rig quandle $Q(15,7)$ is a $7$-almost quasigroup quandle. Furthermore, by Corollary \ref{Corollary 3.4} $\textrm{tor}H_{n}(Q(15,2))$ is annihilated by $3(\textrm{lcm}(15,15-3))=180.$ We have\footnote{We would like to thank L.~Vendramin for computing the homology $H_{3}^{R}(Q(15,2))=\mathbb{Z}\oplus\mathbb{Z}_{2}^{3}\oplus\mathbb{Z}_{30}.$} $H_{2}^{Q}(Q(15,2))=\mathbb{Z}_{2}\oplus\mathbb{Z}_{2}$ and $H_{3}^{Q}(Q(15,2))=\mathbb{Z}_{2}\oplus\mathbb{Z}_{30},$ therefore it is not annihilated by any power of $|X|=15.$ This is possible because the Rig quandle $Q(15,2)$ is not a homogeneous quandle, so the result in \cite{L-N} does not apply. The group $\textrm{Inn}(Q(15,2))$ has $60$ elements, so it annihilates  $H_{2}^{Q}(Q(15,2))$ and $H_{3}^{Q}(Q(15,2)).$}
\end{example}

\begin{table}[h]
  \centering
  \caption{A $3$-almost quasigroup quandle of order $12$; $Q(12,10)$ }\label{Table 3}
  \begin{tabular}{c|cccccccccccc}
 $\ast$ & 1 & 2 & 3 & 4 & 5 & 6 & 7 & 8 & 9 & 10 & 11 & 12 \\
 \hline
 1 & 1 & 1 & 1 & 12 & 11 & 10 & 5 & 4 & 6 & 9 & 7 & 8 \\
 2 & 2 & 2 & 2 & 11 & 10 & 12 & 6 & 5 & 4 & 8 & 9 & 7 \\
 3 & 3 & 3 & 3 & 10 & 12 & 11 & 4 & 6 & 5 & 7 & 8 & 9 \\
 4 & 8 & 9 & 7 & 4 & 4 & 4 & 10 & 12 & 11 & 3 & 2 & 1 \\
 5 & 7 & 8 & 9 & 5 & 5 & 5 & 11 & 10 & 12 & 2 & 1 & 3 \\
 6 & 9 & 7 & 8 & 6 & 6 & 6 & 12 & 11 & 10 & 1 & 3 & 2 \\
 7 & 11 & 12 & 10 & 3 & 1 & 2 & 7 & 7 & 7 & 4 & 5 & 6 \\
 8 & 12 & 10 & 11 & 1 & 2 & 3 & 8 & 8 & 8 & 5 & 6 & 4 \\
 9 & 10 & 11 & 12 & 2 & 3 & 1 & 9 & 9 & 9 & 6 & 4 & 5 \\
 10 & 6 & 5 & 4 & 7 & 8 & 9 & 3 & 2 & 1 & 10 & 10 & 10 \\
 11 & 5 & 4 & 6 & 9 & 7 & 8 & 1 & 3 & 2 & 11 & 11 & 11 \\
 12 & 4 & 6 & 5 & 8 & 9 & 7 & 2 & 1 & 3 & 12 & 12 & 12
\end{tabular}
\end{table}

\section{Future research}
We plan to work on the problem of when the order of the inner automorphism group of a quandle annihilates the torsion of its homology, i.e. $|\textrm{Inn}(X)|\textrm{tor}H_{n}(X)=0.$ We have several examples for which it is the case:

\begin{example}
\begin{enumerate}
  \item \emph{Finite quasigroup quandles. \\ For a finite quasigroup quandle $X,$ $|X|$ annihilates $\textrm{tor}H_{n}(X)$ by Theorem \ref{Theorem 1.8} and Corollary \ref{Corollary 1.9}. Furthermore, since $X$ embeds in $\textrm{Inn}(X),$ we know that $|X|$ divides $|\textrm{Inn}(X)|$ by the orbit-stabilizer theorem. Therefore $|\textrm{Inn}(X)|\textrm{tor}H_{n}(X)=0.$  In particular, for odd $k,$ $\textrm{tor}H_{n}(R_{k})$ is annihilated by $k$ because the dihedral quandle $R_{k}$ is a quasigruop quandle.}
  \item \emph{Rig quandles $Q(6,1),$ $Q(6,2),$ $Q(12,8),$ and $Q(12,9).$\\ Since $Q(6,1)$ and $Q(6,2)$ are $2$-almost quasigroup quandles and $Q(12,8)$ and $Q(12,9)$ are $4$-almost quasigroup quandles, we see that $\textrm{tor}H_{n}(Q(6,1))$ and $\textrm{tor}H_{n}(Q(6,2))$ are annihilated by $24$ and $\textrm{tor}H_{n}(Q(12,8))$ and $\textrm{tor}H_{n}(Q(12,9))$ are annihilated by $96$ by Theorem \ref{Theorem 3.2} and Corollary \ref{Corollary 3.3}. Moreover, $\textrm{Inn}(Q(6,1))$ and $\textrm{Inn}(Q(6,2))$ are isomorphic to the symmetric group $S_{4}$ of order $24$ and $|\textrm{Inn}(Q(12,8))|=96=|\textrm{Inn}(Q(12,9))|$ by \cite{Cla}. Therefore, in all cases described above the order of the inner automorphism group of each quandle annihilates the torsion of its homology.
   }
\end{enumerate}
\end{example}

It is still an open problem whether the order of the quandle inner automorphism group of a quandle annihilates the torsion subgroup of its rack and quandle homology for every finite connected quandle. We can make this problem more general, i.e. we ask whether for any finite quandle $X$, $\textrm{tor}H_{n}(X)$ is annihilated by $|\textrm{Inn}(X)|.$ So far, we can show that if $k$ is odd, then $k$ annihilates the torsion of rack and quandle homology of $R_{2k}$ (compare Conjecture $14$ in \cite{N-P-2}). Notice that $R_{2k}$ is a non-connected quandle and $\textrm{Inn}(R_{2k})$ is isomorphic to the dihedral group $D_{k}$ of order $2k.$

The smallest connected quandle for which our method does not work is the Rig quandle $Q(8,1)$ which is the abelian extension of the Alexander quandle $Q(4,1)=\mathbb{Z}_{2}[t]/(t^2+t+1).$ The order of the quandle inner automorphism group of $Q(8,1)$ is $24$; $H_{2}^{Q}(Q(8,1))=0$ and $H_{3}^{Q}(Q(8,1))=\mathbb{Z}_{8}.$

\section{Acknowledgements}

J\'ozef~H.~Przytycki was partially supported by Simons Collaboration Grant-316446.\\
Seung Yeop Yang was supported by the George Washington University fellowship.

\vspace{3.5mm}
J\'ozef H. Przytycki\\
Department of Mathematics,\\
The George Washington University,\\
Washington, DC 20052 and\\
University of Gda\'nsk\\
e-mail: {\tt przytyck@gwu.edu}\\ \ \\
Seung Yeop Yang\\
Department of Mathematics,\\
The George Washington University,\\
Washington, DC 20052\\
e-mail: {\tt syyang@gwu.edu}


\begin{thebibliography}{99}


\bibitem[AG]{AG}
N.~Andruskiewitsch, M.~Gra\~na, From racks to pointed Hopf algebras, {\it Adv. Math.}, 178, 2003, 177-243;\\ e-print:\ {\tt arXiv:math/0202084 [math.QA]}

\bibitem[CENS]{CENS}
J.~S.~Carter, M.~Elhamdadi, M.~A.~Nikiforou, M.~Saito, Extensions of quandles and cocycle knot invariants, {\it Journal of Knot Theory and Its Ramifications}, 12(6), 2003, 725-738; e-print:\ {\tt arXiv:math/0107021 [math.GT]}

\bibitem[CHNS]{CHNS}
J.~S.~Carter, A.~Harris, M.~A.~Nikiforou, M.~Saito, Cocycle knot invariants, quandle extensions, and Alexander matrices, Suurikaisekikenkyusho Koukyuroku, (Seminar note at RIMS, Kyoto), 1272, 2002, 12-35, available at http://xxx.lanl.gov/math/abs/GT0204113.

\bibitem[CJKLS]{CJKLS}
J.~S.~Carter, D.~Jelsovsky, S.~Kamada, L.~Langford, M.~Saito, State-sum invariants of knotted curves and surfaces from quandle cohomology, {\it Electron. Res. Announc. Math. Sci.}, 5 (1999) 146-156.

\bibitem[CJKS]{CJKS}
J.~S.~Carter, D.~Jelsovsky, S.~Kamada, M.~Saito,
Quandle homology groups, their Betti numbers, and virtual knots,
{\it J. Pure Appl. Algebra} \textbf{157}, 2001, 135-155;\
e-print:\ {\tt http://front.math.ucdavis.edu/math.GT/9909161}

\bibitem[CKS-1]{CKS-1}
J.~S.~Carter, S.~Kamada, M.~Saito, Geometric Interpretations of Quandle Homology, {\it Journal of Knot Theory and Its Ramifications}, 10(3), 2001, 345-386;
e-print:\ {\tt arXiv:math/0006115 [math.GT]}

\bibitem[CKS-2]{CKS-2}
J.~S.~Carter, S.~Kamada, M.~Saito, Surfaces in 4-space, {\it Encyclopaedia
of Mathematical Sciences}, Low-Dimensional Topology III, R.V.Gamkrelidze,
V.A.Vassiliev, Eds., Springer-Verlag, 2004, 213pp.

\bibitem[CSV]{CSV}
W.~E.~Clark, M.~Saito, L.~Vendramin, Quandle coloring and cocycle invariants of composite knots and abelian extensions, {\it MCOM}, to appear; e-print:\ {\tt arXiv:1407.5803 [math.GT]}

\bibitem[Cla]{Cla}
F.~J.~B.~J.~ Clauwens, Small connected quandles,
Preprint, 2011;
e-print:\ {\tt arXiv:1011.2456 [math.GR]}

\bibitem[E-G]{E-G}
P.~Etingof, M.~Gra\~na, On rack cohomology,
{\it J. Pure Appl. Algebra}, 177, 2003, 49-59;\\
e-print:\ {\tt arXiv:math/0201290v2 [math.QA]}

\bibitem[FRS]{FRS}
R.~Fenn, C.~Rourke, B.~Sanderson, An introduction to species and the rack space. Topics in knot theory (Erzurum, 1992), 33-55, {\it NATO Adv. Sci. Inst. Ser. C Math. Phys. Sci.}, 399, Kluwer Acad. Publ., Dordrecht, 1993.

\bibitem[Joy]{Joy} D.~Joyce, A classifying invariant of knots: the knot
quandle, {\it J. Pure Appl. Algebra} \textbf{23}, 1982, 37-65.

\bibitem[L-N]{L-N}
R.~A.~Litherland, S.~Nelson, The Betti numbers of some finite racks,
{\it J. Pure Appl. Algebra}, 178, 2003, 187-202;\\
e-print:\ {\tt arXiv:math/0106165v4 [math.GT] }

\bibitem[Matv]{Matv}
S.~Matveev, Distributive groupoids in knot theory, {\it Matem. Sbornik}, 119(161)(1), 78-88, 1982 (in Russian);
English Translation in {\it Math. USSR-Sbornik}, 47(1), 73-83, 1984.

\bibitem [N-P-1]{N-P-1}
M.~Niebrzydowski, J.~H.~Przytycki,
Homology of dihedral quandles,
{\it J. Pure Appl. Algebra}, \textbf{213}, 2009, 742-755;
e-print:\ {\tt  http://front.math.ucdavis.edu/math.GT/0611803}

\bibitem[N-P-2]{N-P-2}
M.~Niebrzydowski, J.~H.~Przytycki, Homology operations on homology of quandles, {\it Journal of Algebra}, 324, 2010, pp. 1529-1548;
e-print:\ {\tt  http://front.math.ucdavis.edu/0907.4732}

\bibitem[P-Y]{P-Y}
J.~H.~Przytycki, S.~Y.~Yang, The torsion of a finite quasigroup quandle is annihilated by its order, {\it Journal of Pure and Applied Algebra}, 219(10), 2015, 4782-4791; e-print:\ {\tt  arXiv:1411.1727 [math.GT] }

\bibitem[Tak]{Tak}
M.~Takasaki, Abstraction of symmetric transformation, (in Japanese)
{\it Tohoku Math. J.}, 49, 1942/3, 145-207; translation to English
is being prepared by S.~Kamada.

\bibitem[Ven]{Ven}
L.~Vendramin, Rig - A GAP package for racks and quandles. Available at
http://github.com/vendramin/rig/.

\end{thebibliography}
\end{document}